\documentclass[11pt,twoside, leqno]{article}

\usepackage{amssymb,amsmath,amsfonts,amsthm,color,mathrsfs}
\usepackage[Symbol]{upgreek}
\usepackage{txfonts}
\usepackage[nottoc,notlot,notlof]{tocbibind}
\usepackage[active]{srcltx}
\usepackage{hyperref}

\allowdisplaybreaks
\def\rr{{\mathbb R}}
\def\rn{{{\rr}^n}}

\def\cn{{\mathbb N}}

\def\D{{\mathscr D}}

\def\supp{{\mathop\mathrm{\,supp\,}}}

\def\loc{{\mathop\mathrm{\,loc\,}}}

\def\gz{{\gamma}}

\def\E{\mathscr{E}}
\def\L{\mathcal{L}}

\def\r{\right}
\def\lf{\left}

\newtheorem{thm}{Theorem}[section]
\newtheorem{lem}[thm]{Lemma}
\newtheorem{prop}[thm]{Proposition}
\newtheorem{cor}[thm]{Corollary}

\numberwithin{equation}{section}

\begin{document}
\arraycolsep=1pt
\author{Renjin Jiang \& Fanghua Lin}
\title{{\bf Riesz transform under perturbations via heat kernel regularity}
 \footnotetext{\hspace{-0.35cm} 2010 {\it Mathematics
Subject Classification}. Primary  58J35; Secondary  58J05; 35B65; 35K05; 42B20.
\endgraf{
{\it Key words and phrases: Riesz transform, harmonic functions, heat kernels, perturbation}
\endgraf}}}
\maketitle

\begin{center}
\begin{minipage}{11.5cm}\small
{\noindent{\bf Abstract}. Let $M$ be a complete non-compact Riemannian manifold. In this paper,
we derive sufficient conditions on metric perturbation for stability of $L^p$-boundedness of the Riesz transform, $p\in (2,\infty)$.
 We also provide counter-examples regarding in-stability for
$L^p$-boundedness of Riesz transform.
}\end{minipage}
\end{center}
\vspace{0.2cm}
%\tableofcontents

\section{Introduction}
\hskip\parindent
Let $M$ be a complete, connected and non-compact $n$-dimensional  Riemannian manifold, $n\ge 2$.
In this paper, we study the behavior of the Riesz transform under metric perturbations.
As a main tool and also a byproduct, we also obtain stability and instability
of gradient estimates of harmonic functions and heat kernels under metric perturbation.

Let $g_0$ and $g$ be two Riemannian metrics on $M$. Let $\mu_0$, $\mu$, $\mathcal{L}_{0}$, $\mathcal{L}$, $\nabla_0$, $\nabla$,
$\mathrm{div}_0$, $\mathrm{div}$,
be the corresponding Riemannian volumes, non-negative Laplace-Beltrami operators, Riemannian gradient operators and divergence operators,
generated by $g_0$ and $g$, respectively.

Suppose that $g_0$ and $g$ are comparable on $M$, i.e., there exist $C\ge 1$
 such that for any $x\in M$ and ${\bf {v}}=(v_1,\cdots, v_n)\in T_x(M)$ it holds
 \begin{equation}\label{quasi-isometry}
 C^{-1}g^{ij}v_iv_j\le g_0^{ij}v_iv_j\le Cg^{ij}v_iv_j.
\end{equation}
Then a natural question is: if the Riesz operator $\nabla_0\L^{-1/2}_0$ is bounded on $L^p(M,\mu_0)$,
$p\in (1,\infty)$, is $\nabla\L^{-1/2}$ also bounded on $L^p(M,\mu)$?

Note that the case $p=2$ is trivially true as the Riesz operators $\nabla_0\L^{-1/2}_0$ and $\nabla\L^{-1/2}$ are isometries on
$L^2(M,\mu_0)$ and $L^2(M,\mu)$, respectively. For the case $p\in (1,2)$, it was shown by
Coulhon and Duong \cite{cd99} that the Riesz operator is weakly $L^1$-bounded under a doubling
condition and a Gaussian upper bound for the heat kernel. The $L^p$-boundedness then follows from an
interpolating, for all $p\in (1,2)$.

Let $B_0(x,r)$, $B(x,r)$ be open balls induced by the metrics $g_0,\,g$ respectively,  and
$V_0(x,r)$ and $V(x,r)$ the volumes $\mu_0(B_0(x,r))$ and $\mu(B(x,r))$ respectively.
We say that $(M,g_0)$ satisfies a doubling condition, if the exists $C_D>0$ such that for any $x\in M$
and for all $r>0$ that
$$V_0(x,2r)\le C_D V_0(x,r),\leqno(D)$$
and that the heat kernel $p^0_t(x,y)$ of $e^{-t\L_0}$ satisfies a Gaussian upper bound, if there exists $c,C>0$ such that
$$p^0_t(x,y)\le
  \frac{C}{V(x,{\sqrt t})}\exp\lf\{-c\frac{d^2(x,y)}{t}\r\}, \forall\,t>0 \ \& \ \forall\ x,y\in M. \leqno(GUB)
$$
According to \cite{bcs15}, under $(D)$, $(GUB)$ is equivalent to a version of the Sobolev-Poincare
inequality that there exists $q>2$ such that for every ball $B_0(x,r)$ and
each $f\in C^\infty_0(B_0(x,r))$,
$$\left(\fint_{B_0(x,r)}|f|^q\,d\mu_0\right)^{2/q}\le
C_{LS}\left(\fint_{B_0(x,r)}|f|^2\,d\mu_0+r^2\fint_{B_0(x,r)}|\nabla_0 f|^2\,d\mu_0\right).\leqno(SI)$$
Above and in what follows, for a measurable set $\Omega$, $\fint_{\Omega}g\,d\mu_0$ denotes the average of the integrand over it.

It is easy to see that $(D)$ and $(SI)$ are invariant under quasi-isometries. Therefore,
if $(D)$ and $(GUB)$ are satisfied on $(M,g_0)$, then they are also satisfied on $(M,g)$, and
the Riesz operator $\nabla \L^{-1/2}$ is bounded on $L^p(M,\mu)$ for all $p\in (1,2)$.

The case $p>2$ is more involved.
It was shown in \cite{cjks16}  that on a Riemannian manifold $(M,g)$, under $(D)$ together with a scale invariant
 $L^2$-Poincar\'e inequality,
$\nabla \mathcal{L}^{-1/2}$ is bounded on $L^p(M,\mu)$, $p\in (2,\infty)$, if and only if, it holds for every ball $B(x,r)$ and
every solution to $\L u=0$ on $B(x,r)$ that
$$\left(\fint_{B(x,r)}|\nabla u|^p\,d\mu\right)^{1/p}\le \frac{C}{r}\fint_{B(x,r)}|u|\,d\mu.\leqno(RH_{p})$$
See \cite{shz05} for the case $\rn$, and \cite{ac05,acdh,ji17} for earlier results and also further generalizations.
Further, it was shown in \cite{cjks18} that, the local Riesz transform $\nabla (1+\mathcal{L})^{-1/2}$ is bounded on $L^p(M,\mu)$,
$p\in (2,\infty)$, if and only if, the above inequality $(RH_p)$ holds for all balls $B(x,r)$ with $r<1$.

By the perturbation result of Caffarelli and Peral \cite{cp98}, one has a good understanding of the
local gradient estimates for elliptic equations on $\rn$. In particular, for a uniformly elliptic operator $L=-\mathrm{div_\rn}A\nabla_\rn$, if $A$ is uniformly continuous, then \cite{cp98} implies that
any $L$-harmonic functions satisfies $(RH_p)$  on small balls $B(x,r)$ with $r<1$ for all $p<\infty$. This gives the
$L^p$-boundedness of the local Riesz transform $\nabla_{\rn} (1+L)^{-1/2}$ for all $p\in (2,\infty)$.

Then how about the Riesz transform $\nabla_{\rn} L^{-1/2}$?
It was well-known that, for any $p>2$, there exists a uniformly elliptic operator (Meyer's  conic Laplace operator) on $\rr^2$ such that the Riesz transform is not bounded on $L^p(\rr^2)$; see \cite[p.120]{at98} and also Section \ref{example}.
Noting that the conic Laplace operators do not enjoy smoothness at the origin, one may wonder
what happens if the coefficients are smooth?
We however have the following example.
\begin{prop}\label{counter-example}
For any given $p>2$ and $n\ge 2$, there exists a $C^\infty(\rn)$ matrix $A(x)$ satisfying uniformly elliptic condition,
$$c|\xi|^2\le \langle A\xi,\xi\rangle\le C|\xi|^2, \ \forall\ \xi\in\rn,$$
and each order of gradients of $A(x)$ is bounded, such that the Riesz operator
$\nabla_{\rn} L^{-1/2}$, $L=-\mathrm{div_\rn}A\nabla_\rn$, is not bounded on $L^p(\rr^n)$.
\end{prop}
The above proposition implies that apart from the local smoothness (local regularity), some
global controls of the perturbation are needed for the stability issue. { For some
related results we refer to a study of (asymptotically) conic elliptic operators in \cite{lin96} and
other examples motivated by the theory of elliptic homogenization (cf. \cite{AL87b} and \cite[Further
Remarks]{AL87}). In fact one can construct (with some extra work) examples of uniformly regular,
uniformly elliptic operators with isotropic coefficient matrices $a(x)I_n$ so that the conclusion of
the above proposition remains valid.}

In what follows, we shall use the Einstein summation convention for repeated indexes, and $\delta_{ik}$ the
 Kronecker delta function. Our main result reads as follows.
\begin{thm}\label{main-manifold}
Let $g_0,\,g$ be two Riemannian metrics on $M$. Assume  that $g_0$ and $g$ are comparable and
there exists $\epsilon>0$ such that
$$\fint_{B_0(x,r)}|\delta_{ik}-g_0^{ij}g_{jk}|\,d\mu_0\le Cr^{-\epsilon},\ \forall \,1\le i,k\le n,\,\forall x\in M\,\&\,\forall \,r>1.\leqno(GD)$$
Suppose that $(M,g)$ satisfies $(D)$ and $(GUB)$.
Then if for some $p_0\in (2,\infty)$, $\nabla_0 \mathcal{L}_{0}^{-1/2}$ is bounded on
$L^{p}(M,\mu_0)$ and $\nabla (\mathcal{L}+1)^{-1/2}$ is bounded on $L^{p}(M,\mu)$ for all $p\in (2,p_0)$,
$\nabla \mathcal{L}^{-1/2}$ is  bounded on $L^p(M,\mu)$ for all $p\in (2,p_0)$.
\end{thm}

Some remarks are in order. First, the above result can not be true if $\epsilon=0$ as indicated by
the Proposition \ref{counter-example} though one may replace the algebraic decay condition by
a Dini-type condition. Next, if we strengthen the assumption $(GUB)$ to two sides bounds
of the heat kernel then we can include the endpoint that $p = p_0$; see Theorem \ref{metric-perturbation-ly} below.
Moreover, as the  $L^p$-boundedness of the Riesz operator $\nabla \L^{-1/2}$ implies
$$\|\nabla e^{-t\L}\|_{L^p(M,\mu)\to L^p(M,\mu)}\le \frac{C}{\sqrt t}, \ \forall\,t>0,$$
by \cite{acdh,cjks16},
 one further sees that the gradient estimates for heat kernels and harmonic functions are also stable under
such metric perturbations.

Coulhon-Dungey \cite{cdn07} has addressed the stability issue of Riesz transform under perturbations.
In \cite{cdn07}, no assumptions on the volume growth or the upper bound of the heat kernel were required.
However, they assumed the ultra-contractivity, i.e.,
$$\|e^{-t\L}\|_{L^1(M,\mu)\to L^\infty(M,\mu)}\le Ct^{-D/2}, \ t\ge 1;$$
and that $\delta_{ik}-g_0^{ij}g_{jk}\in L^q(M,\mu_0)$ for some $q\in [1,\infty)$ instead of $(GD)$;
see \cite[Theorem 4.1]{cdn07}.
In the case of Riemannian manifolds with lower Ricci curvature bound, the ultra-contractivity requires the volume of unit balls
is non-collapsing, i.e., $\inf_{x\in M} V(x,1)>0$; see \cite[Proposition 3.1]{ji15}.
By \cite{ck88}, there are Riemannian manifolds with non-negative Ricci curvature, on which the volume
of unit balls does collapse.  Moreover, if $(D)$ and $\inf_{x\in M} V(x,1)>0$  hold, then $\delta_{ik}-g_0^{ij}g_{jk}\in L^q(M,\mu_0)$, $q\in [1,\infty)$, implies $(GD)$.

{ Recently, Blank, Le Bris and Lions \cite{bbl18a,bbl18b} addressed the issue of perturbations
related to the elliptic homogenization, their results are rather interesting in comparison with that
of \cite{cdn07} for the case that $(M_0,h)$ being Euclidean spaces with a nice periodic metric $h$.
Instead of the ultra-contractivity property as described above, \cite{bbl18a,bbl18b} used a
continuity argument starting from the estimates established in \cite{AL91}. We also
note that most of conclusions of \cite{bbl18a,bbl18b} are also true for systems, while our proofs
here and \cite{cdn07} work only for the scalar case.}

Our main achievement here is that we find the condition $(GD)$, which works also for the collapsing
case. In particular results here cover the case of complete manifolds with non-negative Ricci
curvature.
{  In general  $(GD)$ allows a much larger class (than $L^p$ ($p<\infty$)) of
perturbations. For example, in $\rn$, a perturbation along a strip $\rr^{n-1}\times [0,1]$ satisfies
$(GD)$ but is not in $L^p$ for $p<\infty$. }
For the proof, we shall follow the basic strategy of \cite{cdn07}, where the key step is to estimate
the difference of the operator norm
$$\|\nabla[(1+t\mathcal{L}_0)^{-1}-(1+t\mathcal{L})^{-1}]\|_{L^p(M,\mu)\to L^p(M,\mu)}\le Ct^{-\alpha-1/2}, \ t\ge 1,$$
for some $\alpha>0$. If the ultra-contractivity holds, then such an estimate is relatively easy to
establish; see \cite[Proposition 2.2]{cdn07}.
However, without ultra-contractivity, the proof is more involved. The estimates for the heat kernel
and its gradient (cf. \cite{acdh,cjks16}), together with $(GD)$ are essential in our proofs.

Let us list several consequences of Theorem \ref{main-manifold}. It is well known that if
$(M,g)$ has lower Ricci curvature bound, then the local Riesz transform is $L^p$-bounded for all $p\in (1,\infty)$; see \cite{acdh}.
\begin{cor}\label{cor-lowerRicci}
Assume that $g,\,g_0$ are two metrics on $M$, that satisify  \eqref{quasi-isometry} and there exists $\epsilon>0$ such that
$$\fint_{B_0(x,r)}|\delta_{ik}-g_0^{ij}g_{jk}|\,d\mu_0\le Cr^{-\epsilon},\  \forall \,1\le i,k\le n,\,\forall x\in M\,\&\,\forall \,r>1.$$
Suppose that $(M,g)$ has Ricci curvature bounded from below and satisfies $(D)$ and $(GUB)$.
Then if  $\nabla_0 \mathcal{L}_{0}^{-1/2}$ is bounded on $L^{p}(M,\mu_0)$ for all $p\in (2,p_0)$, where  $p_0\in (2,\infty]$,
$\nabla \mathcal{L}^{-1/2}$ is also bounded on $L^p(M,\mu)$ for all $p\in (2,p_0)$.
\end{cor}

Note that in particular any compact metric perturbation satisfies $(GD)$.
\begin{cor}\label{cor-smooth}
Suppose that $(M,g)$ and $(M,g_0)$ satisfy $(D)$ and $(GUB)$, and have Ricci curvature bounded from below.
If $g$ coincides with  $g_0$ outside a compact subset,  then for all $p_0\in (2,\infty]$,
$\nabla_0 \mathcal{L}_{0}^{-1/2}$ is bounded on $L^p(M,\mu_0)$ for all $p\in (2,p_0)$,
if and only if, $\nabla \mathcal{L}^{-1/2}$  is bounded on $L^p(M,\mu)$ for all $p\in (2,p_0)$.
\end{cor}
Carron \cite{ca07} and Devyver \cite{de15} had addressed the question of stability of compact perturbation,
under the validity of global Sobolev inequality instead of $(D)$ and $(GUB)$. The global Sobolev inequality in general is a stronger
requirement than $(D)$ and $(GUB)$; see \cite[Remark 1.1]{de15}.  On the other hand, the changing of
the topology of manifolds is out of reach is this work, which however is allowed in \cite{ca07,de15},
see also \cite{ji17}.

An easy consequence follows for the case of non-negative Ricci curvature.
\begin{cor}\label{cor-nonnegative}
Assume that $(M,g_0)$ has non-negative Ricci curvature.
If $g$ coincides with  $g_0$ outside a compact subset, then
the Riesz transform $\nabla \mathcal{L}^{-1/2}$ is bounded on $L^p(M,\mu)$ for all $p\in (1,\infty)$.
\end{cor}
Zhang \cite{zhang06} had derived a sufficient condition on the perturbation of a manifold with non-negative
Ricci curvature for the stability of Yau's estimate (equivalent to $(RH_p)$ with $p=\infty$, cf. \cite{chy,cjks16,ya75}),
which implies the boundedness of the Riesz transform for all $p\in (1,\infty)$ by \cite[Theorem 1.9]{cjks16}.
We did not prove the stability of Yau's estimate, but the advantage of our result is that our
condition $(GD)$ is much more explicit and, it is convenient for applications.

Let us mention a few examples that our result can apply. Besides manifolds with non-negative Ricci
curvature, conic manifolds (cf. \cite{lh99,lin96}), as well as co-compact covering manifold with
polynomial growth deck transformation group (cf. \cite{dng04a}), Lie groups of polynomial growth (cf.
\cite{Al92,var88}) satisfy the doubling condition $(D)$ and $(GUB)$. Indeed the stronger Li-Yau
estimate is true (see Theorem \ref{metric-perturbation-ly} below). By \cite{gri-sal09}, $(GUB)$ is
preserved under gluing operation; see \cite{cch06,ji17} for studies of Riesz transforms in this
direction. Therefore, our result applies to these settings if the metric perturbation satisfies
$(GD)$. Our result also applies on the Euclidean space $\rr^n$ for elliptic operators (including
degenerate operators); see Theorem \ref{main-denegerate}.

{
Finally let us state a corollary for the Euclidean case. The balls $B(x,r)$ in the following
corollary are induced by the standard Euclidean metric.
\begin{cor}\label{cor-euclidean}
Let $\mathcal{L}_0=-\mathrm{div}_{\rn}(A_0\nabla_{\rn}) $, $\mathcal{L}=-\mathrm{div}_{\rn}(A\nabla_{\rn}) $ be uniformly elliptic operators on $\rn$, $n\ge 2$, with $A_0,A$ being uniformly continuous
on $\rn$.  Suppose that there exists $\epsilon>0$ such that
$$\fint_{B(y,r)} \left|A-A_0\right|\,dx\le \frac{C}{r^\epsilon},\quad\, \forall \,y\in\rn \, \& \, \forall \, r>1.$$
Then for $p\in (2,\infty)$,  $\nabla_{\rn}\mathcal L_0^{-1/2}$ is bounded on $L^p(\rn)$  if and only if  $\nabla_{\rn}\mathcal L^{-1/2}$ is bounded on $L^p(\rn)$.
\end{cor}
}

The paper is organized as follows. In Section \ref{sec-manifold}, we study the case of manifolds and prove Theorem \ref{main-manifold}
and its corollaries. In Section \ref{sec-degenerate} we discuss the case of degenerate elliptic equations on $\rn$.
In Section \ref{example}, we discuss the conic Laplace operators and present the proof of Theorem \ref{counter-example} there.
In Appendix \ref{appendix}, we recall some basic facts regarding the boundedness of functional operators.

\section{Metric perturbation on manifolds}\label{sec-manifold}
\hskip\parindent In this section, we study the behavior of Riesz transform under metric perturbation on manifolds.

Note that as $g,g_0$ are comparable on $M$, the resulting Riemannian volumes $\mu$ and $\mu_0$ are also comparable, which implies that for any $p\in [1,\infty]$
$$L^p(M,\mu)=L^p(M,\mu_0),$$
also the boundedness of $\nabla \mathcal{L}_{0}^{-1/2}$, $\nabla \mathcal{L}^{-1/2}$ on $L^p(M,\mu)$, is equivalent
to the boundedness of $\nabla_0 \mathcal{L}_{0}^{-1/2}$, $\nabla_0 \mathcal{L}^{-1/2}$ on $L^p(M,\mu_0)$, respectively.
In what follows, we shall simply denote by $L^p(M)$ the Lebesgue space $L^p(M,\mu)$ or $L^p(M,\mu_0)$,
and denote by $\|\cdot\|_p$, $\|\cdot\|_{p\to p}$ the $L^p(M)$ norm and
the operator norm $\|\cdot\|_{L^p(M)\to L^p(M)}$, respectively.

As the consequence of  \eqref{quasi-isometry} also, one sees that the condition $(GD)$ is equivalent to
$$\fint_{B_0(x,r)}|\delta_{ik}-g_0^{ij}g_{jk}|\,d\mu_0\sim \fint_{B(x,r)}|\delta_{ik}-g_0^{ij}g_{jk}|\,d\mu\le Cr^{-\epsilon},\
\forall\ 1\le i,k\le n, \forall x\in M\,\&\,\forall \,r>1.$$

Let us outline the proof of Theorem \ref{main-manifold}, which we follow the approach in \cite{cdn07}.
Note that our main ingredients are Proposition \ref{metric-lem-1} and Proposition \ref{metric-lem-2}  below.
\begin{proof}[Proof of Theorem \ref{main-manifold}]
To show that $\nabla \mathcal{L}^{-1/2}$ is bounded on $L^p(M)$ for $p\in (2,p_0)$, it suffices to show that
\begin{equation}\label{metric-main}
\|\nabla \mathcal{L}_0^{-1/2}-\nabla \mathcal{L}^{-1/2}\|_{p\to p}\le C.
\end{equation}

We write
\begin{eqnarray*}
\nabla \mathcal{L}_0^{-1/2}-\nabla \mathcal{L}^{-1/2}&&=\frac{1}{\pi}\int_0^1 \nabla[(1+t\mathcal{L}_0)^{-1}-(1+t\mathcal{L})^{-1}]\frac{\,dt}{\sqrt t}\\
&&\quad+\frac{1}{\pi}\int_1^\infty\nabla[(1+t\mathcal{L}_0)^{-1}-(1+t\mathcal{L})^{-1}]\frac{\,dt}{\sqrt t}.
\end{eqnarray*}
By Lemma \ref{lemma-s1} and that $\nabla \mathcal{L}_{0}^{-1/2}$ and $\nabla (1+\mathcal{L})^{-1/2}$ are bounded on
$L^p(M)$, one has
\begin{equation}\label{metric-7}
\left\|\frac{1}{\pi}\int_0^1 \nabla[(1+t\mathcal{L}_0)^{-1}-(1+t\mathcal{L})^{-1}]\frac{\,dt}{\sqrt t}\right\|_{p\to p}\le C.
\end{equation}

For the remaining term, by the following Proposition \ref{metric-lem-1} and Proposition \ref{metric-lem-2},
we see that there exists $\alpha>0$ such that
\begin{equation}\label{metric-2}
\|\nabla[(1+t\mathcal{L}_0)^{-1}-(1+t\mathcal{L})^{-1}]\|_{p\to p}\le Ct^{-\alpha}\|\nabla(1+t\mathcal{L})^{-1/2}\|_{p\to p}.
\end{equation}
Note here we may take $\alpha=\min\{\frac{\epsilon(p-1)}{2p},\frac{\epsilon(p_0-p)}{2p(p_0+p)}\}=\frac{\epsilon(p_0-p)}{2p(p_0+p)}$ if $p_0<\infty$, and $\alpha=\frac{\epsilon}{2p}$ when $p_0=\infty$.

By using the boundedness of the local Riesz transform
$$\|\nabla(1+\mathcal{L})^{-1/2}\|_{p\to p}\le C,$$
one obtains for any $t>1$ that
$$\|\nabla(1+t\mathcal{L})^{-1}\|_{p\to p}\le
\|\nabla(1+\mathcal{L})^{-1/2}\|_{p\to p}\|(1+\L)^{1/2}(1+t\mathcal{L})^{-1}\|_{p\to p}\le C.$$
This together with Lemma \ref{lemma-s2} implies that
$$\|\nabla(1+t\mathcal{L})^{-1/2}\|_{p\to p}\le C.$$
Inserting this into \eqref{metric-2}, one finds
\begin{equation*}
\|\nabla[(1+t\mathcal{L}_0)^{-1}-(1+t\mathcal{L})^{-1}]\|_{p\to p}\le Ct^{-\alpha},
\end{equation*}
and
\begin{equation}\label{metric-6}
\|\nabla(1+t\mathcal{L})^{-1}]\|_{p\to p}\le Ct^{-\alpha}+\|\nabla(1+t\mathcal{L}_0)^{-1}\|_{p\to p}\le Ct^{-\alpha\wedge 1/2}.
\end{equation}
Above we used the fact
\begin{eqnarray*}
\|\nabla(1+t\mathcal{L}_0)^{-1}\|_{p\to p}&&\le \|\nabla {L}_0^{-1/2}\|_{p\to p}\|{L}_0^{1/2}(1+t\mathcal{L}_0)^{-1}\|_{p\to p}\\
&&\le C\left\|\int_0^\infty {L}_0^{1/2}e^{-s-st\mathcal{L}_0}\,ds \right\|_{p\to p}\\
&&\le C\int_0^\infty \frac{e^{-s}}{\sqrt{st}}\,ds\le Ct^{-1/2}. \\
\end{eqnarray*}
Now inserting \eqref{metric-6} into \eqref{metric-2}, we get
\begin{equation*}
\|\nabla[(1+t\mathcal{L}_0)^{-1}-(1+t\mathcal{L})^{-1}]\|_{p\to p}\le Ct^{-\alpha-\alpha\wedge 1/2},
\end{equation*}
and
\begin{equation*}
\|\nabla(1+t\mathcal{L})^{-1}]\|_{p\to p}\le Ct^{-2\alpha \wedge 1/2}.
\end{equation*}
Repeating this argument finitely times (depending on $\alpha$), we arrive at
\begin{equation}\label{metric-3}
\|\nabla[(1+t\mathcal{L}_0)^{-1}-(1+t\mathcal{L})^{-1}]\|_{p\to p}\le Ct^{-\alpha-1/2}, \forall\,t>1,
\end{equation}
and
\begin{equation}\label{metric-4}
\|\nabla(1+t\mathcal{L})^{-1}]\|_{p\to p}\le Ct^{-1/2}.
\end{equation}
Inserting \eqref{metric-3} into the term II, we conclude that
\begin{eqnarray}\label{metric-8}
\left\|\frac{1}{\pi}\int_1^\infty\nabla[(1+t\mathcal{L}_0)^{-1}-(1+t\mathcal{L})^{-1}]\frac{\,dt}{\sqrt t}\right\|_{p\to p}\le C\int_1^\infty t^{-\alpha-1/2}\frac{\,dt}{\sqrt t}\le C.
\end{eqnarray}
Combining the estimates of \eqref{metric-7} and \eqref{metric-8}, we get
$$\|\nabla \mathcal{L}_0^{-1/2}-\nabla \mathcal{L}^{-1/2}\|_{p\to p}\le C,$$
and hence,
$$\|\nabla \mathcal{L}^{-1/2}\|_{p\to p}\le C,$$
as desired.
\end{proof}

Let us estimate the difference $\nabla[(1+t\mathcal{L}_0)^{-1}-(1+t\mathcal{L})^{-1}]$ for $t>1$ to show
\eqref{metric-2}.

%\begin{thm}\label{grad-heat}
% Assume that $(M,g)$ satisfies $(D)$ and $(GUB)$.
% Let $p_0\in (2,\infty]$. Then the following statements are equivalent:
%
%(i) For any $p\in (2,p_0)$, there exists $\gz>0$ such that
%$$\int_M|\nabla_x p_t(x,y)|^{p}\exp\left\{\gz d^2(x,y)/t\right\}\,d\mu(x)\le \frac{C}{t^{{p}/2}V(y,\sqrt t)^{{p}-1}} \leqno (GLY_{p})$$
%for all $t>0$ and a.e. $y\in X$.
%
%(ii) For any $p\in (2,p_0)$, the gradient of the heat semigroup $|\nabla H_t|$ is bounded on $L^{p}(M)$ for each $t>0$ with
%$$\||\nabla H_t|\|_{p\to p}\le \frac{C}{\sqrt t}.\leqno (G_{p})$$
%\end{thm}

Set $|g|=|\det(g_{ij})|$ and $|g_0|=|\det(g_0)_{ij}|.$
Let us begin with the formula
\begin{eqnarray*}
\L_0-\L&&=\frac{1}{\sqrt{|g|}}\partial_i\left(\sqrt{|g|}g^{ij}\partial_j\right)
-\frac{1}{\sqrt{|g_0|}}\partial_i\left(\sqrt{|g_0|}g_0^{ij}\partial_j\right)\\
&&=\frac{1}{\sqrt{|g|}}\partial_i\left(\left[\sqrt{|g|}g^{ij}-
\sqrt{|g_0|}g_0^{ij}\right]\partial_j\right)-\frac{\sqrt{|g|}-\sqrt{|g_0|}}
{\sqrt{|g_0|}\sqrt{|g|}}\partial_i\left(\sqrt{|g_0|}g_0^{ij}\partial_j\right),
\end{eqnarray*}
and for any $t>0$,
\begin{eqnarray*}
(1+t\mathcal{L})^{-1}-(1+t\mathcal{L}_0)^{-1}=t(1+t\mathcal{L})^{-1}\left(\L_0-\L\right)(1+t\mathcal{L}_0)^{-1}.
\end{eqnarray*}
Set
\begin{eqnarray*}
I_t^1&&= t\nabla(1+t\mathcal{L})^{-1}\left(\frac{1}{\sqrt{|g|}}\partial_i\left(\left[\sqrt{|g|}g^{ij}-
\sqrt{|g_0|}g_0^{ij}\right]\partial_j\right)\right)(1+t\mathcal{L}_0)^{-1}\\
&& =t\nabla(1+t\mathcal{L})^{-1}\left(\frac{1}{\sqrt{|g|}}\partial_i\left(\left[\sqrt{|g|}(\delta_{ik}-g_0^{ij} g_{jk})g^{kj}\right]\partial_j\right)-
\frac{1}{\sqrt{|g|}}\partial_i\left(\left[\sqrt{|g_0|}-\sqrt{|g|}\right]g_0^{ij}\partial_j\right)\right)(1+t\mathcal{L}_0)^{-1}
\end{eqnarray*}
and
\begin{eqnarray*}
II_t^2&&=t\nabla(1+t\mathcal{L})^{-1}\left(\frac{\sqrt{|g|}-\sqrt{|g_0|}}
{\sqrt{|g_0|}\sqrt{|g|}}\partial_i\left(\sqrt{|g_0|}g_0^{ij}\partial_j\right)\right)(1+t\mathcal{L}_0)^{-1}\\
&&= t\nabla(1+t\mathcal{L})^{-1}\left(1-\sqrt{{|g_0|}/{|g|}}\right)\L_0(1+t\mathcal{L}_0)^{-1}.
\end{eqnarray*}

Note that \eqref{quasi-isometry} implies that $C^{-1}|g|\le |g_0|\le C|g|$.
Moreover, from the assumption
$$\fint_{B_0(x,r)}|\delta_{ik}-g_0^{ij} g_{jk}|\,d\mu_0\le Cr^{-\epsilon},\ \forall\ 1\le i,k\le n, \ \forall x\in M\,\&\,\forall \,r>1,$$
for some $\epsilon>0$, one deduces that for all $x\in M$ and all $r>1$
\begin{equation}\label{pol-decay}
\fint_{B_0(x,r)}|1-\det (g_0^{ij} g_{jk})|\,d\mu_0=\fint_{B_0(x,r)}\left|1-\frac{|g|}{|g_0|}\right|\,d\mu_0
\sim \fint_{B(x,r)}\left|1-\frac{|g|}{|g_0|}\right|\,d\mu\le Cr^{-\epsilon}.
\end{equation}

\begin{prop}\label{metric-lem-1}
Assume that $(M,g_0)$ satisfies $(D)$ and $(GUB)$, and that \eqref{quasi-isometry} holds.
Suppose that  there exists $\epsilon>0$ such that
$$\fint_{B_0(x,r)}|\delta_{ik}-g_0^{ij} g_{jk}|\,d\mu_0\le Cr^{-\epsilon},\ \forall\ 1\le i,k\le n, \ \forall x\in M\,\&\,\forall \,r>1.$$
Then for any $p\in (2,\infty)$, there exists $C>0$ such that for each $t>1$
$$\|II_t^2\|_{p\to p}\le Ct^{-\epsilon(p-1)/2p}\|\nabla(1+t\mathcal{L})^{-1/2}\|_{p\to p}.$$
\end{prop}
\begin{proof}
{\bf Step 1.} We claim that it holds for each $t>1$ that
$$\left\|(1+t\mathcal{L})^{-1/2}\left(1-\sqrt{{|g_0|}/
{|g|}}\right)\right\|_{p\to p}\le Ct^{-\epsilon(p-1)/2p}.$$
For any $f\in C^\infty_c(M)$ one has
\begin{eqnarray*}
(1+t\mathcal{L})^{-1/2}\left(1-\sqrt{{|g_0|}/
{|g|}}\right)f(x)&&=C\int_0^\infty e^{-s(1+t\L)}\left(1-{{|g_0|}/
{|g|}}\right)f(x)\frac{\,ds}{\sqrt s}.
\end{eqnarray*}
For the first term, by using the fact $|g|\sim |g_0|$ we conclude that
\begin{eqnarray*}
\left\|\int_0^{1/t}e^{-s(1+t\L)}\left([1-\sqrt{|g_0|/{|g|}}]f\right)(x)\frac{\,ds}{\sqrt s}\right\|_p&&\le \int_0^{1/t}
\left\|e^{-s(1+t\L)}\left((1-\sqrt{|g_0|/|g|})f\right)\right\|_p\frac{\,ds}{\sqrt s}\\
&&\le \int_0^{1/t} e^{-s}\left\|(1-\sqrt{|g_0|/
|g|})f\right\|_p\frac{\,ds}{\sqrt s}\\
&&\le \frac{C}{\sqrt t}\|f\|_p.
\end{eqnarray*}
Let $p_t(x,y)$ denote the heat kernel of $e^{-t\L}$. For the remaining estimate, note that the heat kernel of $e^{-st\L}$ satisfies
$$0<p_{st}(x,y)\le \frac{C}{V(x,\sqrt{st})}e^{-\frac{d(x,y)^2}{cst}},$$
which is a consequence of $(M,g_0)$ satisfying $(GUB)$ and $g\sim g_0$.

Using the H\"older inequality and $|g|\sim |g_0|$ again, one concludes that
\begin{eqnarray*}
&&\int_M p_{st}(x,y)\left(1-\sqrt{|g_0|/
|g|}(y)\right)|f(y)|\,d\mu(y)\\
&&\le \left(\int_M \frac{Ce^{-\frac{d(x,y)^2}{2cst}}}{V(x,\sqrt{st})}|f(y)|^p \,d\mu(y)\right)^{1/p}
\left(\int_M \frac{Ce^{-\frac{d(x,y)^2}{2cst}}}{V(x,\sqrt{st})}\left(1-\sqrt{|g_0|/
|g|}(y)\right)^{p'}\,d\mu(y)\right)^{1/p'}\\
&&\le \left(\int_M \frac{Ce^{-\frac{d(x,y)^2}{2cst}}}{V(x,\sqrt{st})}|f(y)|^p\,d\mu(y)\right)^{1/p}
\left(\sum_{k=1}^\infty \frac{e^{-c2^{2k}}}{V(x,\sqrt{st})}\int_{B(x,2^k\sqrt{st})}\left|1-\sqrt{|g_0|/
|g|}(y)\right|^{p'}\,d\mu(y)\right)^{1/p'}\\
&&\le \left(\int_M \frac{Ce^{-\frac{d(x,y)^2}{2cst}}}{V(x,\sqrt{st})}|f(y)|^p\,d\mu(y)\right)^{1/p}
\left(\sum_{k=1}^\infty \frac{e^{-c2^{k}}}{V(x,2^k\sqrt{st})}\int_{B(x,2^k\sqrt{st})}\left|1-\sqrt{|g_0|/
|g|}(y)\right|\,d\mu(y)\right)^{1/p'}\\
&&\le \frac{C}{(st)^{\epsilon/2p'}}\left(\int_M \frac{e^{-\frac{d(x,y)^2}{2cst}}}{V(x,\sqrt{st})}|f(y)|^p\,d\mu(y)\right)^{1/p},
\end{eqnarray*}
where in the last inequality we used the estimate
\begin{eqnarray}
\frac{1}{V(x,2^k\sqrt{st})}\int_{B(x,2^k\sqrt{st})}\left|\frac{\sqrt{|g|}-\sqrt{|g_0|}}{
\sqrt{|g|}}\right|\,d\mu
&&=\fint_{B(x,2^k\sqrt{st})}\frac{||g|-|g_0||}{\sqrt{|g|}\sqrt{|g|+|g_0|}}\,d\mu\nonumber\\
&&\le C\fint_{B(x,2^k\sqrt{st})}\frac{||g|-|g_0||}{|g_0|}\,d\mu\le \frac{C}{(2^k\sqrt{st})^{\epsilon}},
\end{eqnarray}
where the last estimate follows from \eqref{pol-decay}.
Using this, one deduces that
\begin{eqnarray*}
&&\left\|\int_{1/t}^\infty e^{-s(1+t\L)}\left(\left|1-\sqrt{|g_0|/
|g|}(y)\right|f\right)(x)\frac{\,ds}{\sqrt s}\right\|_p\\
&&\le C\int_{1/t}^\infty \frac{C}{(st)^{\epsilon/2p'}} e^{-s}\left(\int_M  \int_M \frac{1}{V(x,\sqrt{st})}e^{-\frac{d(x,y)^2}{2cst}}|f(y)|^p\,d\mu(y) \,d\mu(x)
\right)^{1/p}\frac{\,ds}{\sqrt s}\\
&&\le C\|f\|_{p}\int_{1/t}^\infty \frac{C}{(st)^{\epsilon/2p'}} e^{-s}\frac{\,ds}{\sqrt s}\\
&&\le Ct^{-\epsilon/2p'}\|f\|_{p},
\end{eqnarray*}
where $\epsilon\in (0,1)$, $p'\in (1,2)$. This and the estimate for the first term completes the proof of   {\bf Step 1}.

{\bf Step 2.}
Noticing  that
$$\|\L_0(1+t\L_0)^{-1}\|_{p\to p}=1/t\left\|1-(1+t\L_0)^{-1}\right\|_{p\to p}\le C/t,$$
which together with the first step gives that
\begin{eqnarray*}
\|II_t^2\|_{p\to p}\le Ct^{-\epsilon/2p'}\|\nabla(1+t\mathcal{L})^{-1/2}\|_{p\to p},
\end{eqnarray*}
which completes the proof.
\end{proof}

Recall that $\nabla,\nabla_0$, $\mathrm{div},\mathrm{div}_0$ are Riemannian gradients and divergences
induced by $g,\,g_0$, respectively.
\begin{prop}\label{metric-lem-2}
Assume that $(M,g_0)$ satisfies $(D)$ and $(GUB)$, and that \eqref{quasi-isometry} holds.
Suppose that  $\nabla \mathcal{L}_{0}^{-1/2}$ is bounded on
$L^{p_0}(M)$ for some $p_0\in (2,\infty)$, and there exists $\epsilon>0$ such that
$$\fint_{B_0(x,r)}|\delta_{ik}-g_0^{ij}g_{jk}|\,d\mu_0\le Cr^{-\epsilon},\ \forall\ 1\le i,k\le n, \ \forall x\in M\,\&\,\forall \,r>1.$$
Then for each $p\in (2,p_0)$ there exists $C>0$ such that for each $t>1$
\begin{eqnarray*}
\|I_t^1\|_{p\to p}\le Ct^{-\frac{\epsilon(p_0-p)}{2p(p_0+p)}} \|\nabla(1+t\mathcal{L})^{-1/2}\|_{p\to p}.
\end{eqnarray*}
\end{prop}
\begin{proof} Set $A=\{a_{ik}\}_{1\le i,k\le n}=\{g_0^{ij}g_{jk}\}_{1\le i,k\le n}$.
For simplicity of notions, we represent $I_t^1$ in term of Riemannian gradient and divergence as
\begin{eqnarray*}
I_t^1&& =t\nabla(1+t\mathcal{L})^{-1}\left(\frac{1}{\sqrt{|g|}}\partial_i\left(\left[\sqrt{|g|}(\delta_{ik}-g_0^{il}g_{lk})g^{kj}\right]\partial_j\right)-
\frac{1}{\sqrt{|g|}}\partial_i\left(\left[\sqrt{|g_0|}-\sqrt{|g|}\right]g_0^{ij}\partial_j\right)\right)(1+t\mathcal{L}_0)^{-1}\\
&&= t\nabla(1+t\mathcal{L})^{-1} \mathrm{div}\left((I-A)\nabla -
\left|\sqrt{|g_0|/|g|}-1\right|\nabla_0\right)(1+t\mathcal{L}_0)^{-1}.
\end{eqnarray*}

{\bf Step 1}. Noting that $(M,g_0)$ satisfies $(D)$ and $(GUB)$, and that \eqref{quasi-isometry} holds,
$(M,g)$ also satisfies $(D)$ and $(GUB)$. It follows from \cite{cd99} that
$\nabla \mathcal{L}^{-1/2}$ and $\nabla \mathcal{L}_0^{-1/2}$ are bounded on $L^q(M)$ for all $q\in (1,2)$.

Since $(1+t\mathcal{L})^{-1/2}\mathrm{div}$ is the dual operator of $\nabla (1+t\mathcal{L})^{-1/2}$,
and $\nabla \mathcal{L}^{-1/2}$  is bounded on $L^{p'}(M)$, $p'\in (1,2)$,
one has that
$$ \left\|(1+t\mathcal{L})^{-1/2}\mathrm{div}\right\|_{p\to p}=
\left\|\nabla (1+t\mathcal{L})^{-1/2}\right\|_{p'\to p'}\le C/\sqrt t, \ \forall\ t>1.$$

{\bf Step 2}. We claim that it holds
$$\left\|\left((I-A)\nabla -
\left|\sqrt{|g_0|/|g|}-1\right|\nabla_0\right)(1+t\mathcal{L}_0)^{-1}\right\|_{p\to p}\le Ct^{-1/2-\epsilon(p_0-p)/(2p(p_0+p))},\ \forall\ t>1.$$
For any $f\in C^\infty_c(M)$, let us first show that
$$\left\|\left((I-A)\nabla\right)(1+t\mathcal{L}_0)^{-1}f\right\|_{p}\le Ct^{-1/2-\epsilon(p_0-p)/(2p(p_0+p))}\|f\|_{p},$$
the other term can be estimated similarly. We write
\begin{eqnarray*}
&&\left\|\left((I-A)\nabla\right)(1+t\mathcal{L}_0)^{-1}f\right\|_{p}\\
&&\quad\le \int_0^{1/t} \|\left((I-A)\nabla\right) e^{-s(1+t\mathcal{L}_0)}f\|_{p}\,ds+\int_{1/t}^\infty \|\left((I-A)\nabla\right) e^{-s(1+t\mathcal{L}_0)}f\|_{p}\,ds.
\end{eqnarray*}
As
$$\|\nabla \L_0^{-1/2}\|_{p_0\to p_0}\le C,$$
one has
$$\|\nabla \L_0^{-1/2}\|_{p\to p}\le C$$
for any $p\in (2,p_0)$, and hence for any $s>0$ that
$$\|\nabla e^{-s\L_0}\|_{p\to p}\le Cs^{-1/2}. \leqno(G_p)$$
Therefore by using $(G_p)$ and that $A$ is bounded, one has the estimate
\begin{eqnarray}\label{3.11}
\int_0^{1/t} \|\left((I-A)\nabla\right) e^{-s(1+t\mathcal{L}_0)}f\|_{p}\,ds&&\le C\int_0^{1/t} \|\nabla e^{-s(1+t\mathcal{L}_0)}f\|_{p}\,ds\nonumber\\
&&\le C\int_0^{1/t} \frac{e^{-s} }{\sqrt{st}}\|f\|_{p}\,ds\nonumber\\
&&\le Ct^{-1}\|f\|_{p}.
\end{eqnarray}
The estimate of the remaining integrand over $(1/t,\infty)$ is more involved.
By using the boundedness of $\nabla \L_0^{-1/2}$ on $L^{p_0}(M)$ and \cite[Proposition 1.10]{acdh},
we see that, for any $p\in (2,p_0)$ there exist $C,\gamma_p>0$ such that for all $t>0$ and $y\in M$
$$\int_M|(\nabla_0)_x p^0_t(x,y)|^{p}\exp\left\{\gz_p d^2(x,y)/t\right\}\,d\mu_0(x)\le \frac{C}{t^{{p}/2}V_0(y,\sqrt t)^{{p}-1}},\leqno(GLY_{p})$$
where $p^0_t(x,y)$ denotes the heat kernel of  $e^{-t\mathcal{L}_0}$.
By using \eqref{quasi-isometry} that $g\sim g_0$, we see that $(GLY_{p})$ is equivalent to
$$\int_M|\nabla_x p^0_t(x,y)|^{p}\exp\left\{\gz_p d^2(x,y)/t\right\}\,d\mu(x)\le \frac{C}{t^{{p}/2}V_0(y,\sqrt t)^{{p}-1}}.$$
In what follows we shall not distinguish these two estimate.

Let $\gz\in (0,\gz_0)$ to be fixed later.
By using the H\"older inequality, one sees that
\begin{eqnarray}\label{3.12}
|\nabla e^{-st\mathcal{L}_0}f(x)|&&\le C\int_M |\nabla_x p^0_{st}(x,y)||f(y)|\,d\mu(y)\nonumber\\
&&\le C\left(\int_M |\nabla_x p^0_{st}(x,y)|^p\exp\left\{\frac{\gz d^2(x,y)}{st}\right\}V_0(y,\sqrt t)^{{p}-1}|f(y)|^p\,d\mu(y)\right)^{1/p}\nonumber\\
&&\quad\times
\left(\int_M V_0(y,\sqrt {st})^{-1}\exp\left\{-\frac{c(p)\gz d^2(x,y)}{st}\right\}\,d\mu(y)\right)^{1/p'}\nonumber\\
&&\le C\left(\int_M |\nabla p^0_{st}(x,y)|^p\exp\left\{\frac{\gz d^2(x,y)}{st}\right\}V_0(y,\sqrt t)^{{p}-1}|f(y)|^p)\,d\mu(y)\right)^{1/p}.
\end{eqnarray}
Above in the last inequality we used the doubling condition to conclude that for any $x,y\in M$ and any $r>0$ that
$$V_0(y,r)^{-1}\le CV_0(x,r+d(x,y))^{-1}\left(\frac{r+d(x,y)}{r}\right)^\Upsilon\le CV_0(x,r)^{-1}\left(\frac{r+d(x,y)}{r}\right)^\Upsilon$$
for some $\Upsilon>0$, and therefore
\begin{eqnarray*}
&&\int_M V_0(y,\sqrt {st})^{-1}\exp\left\{-\frac{c(p)\gz d^2(x,y)}{st}\right\}\,d\mu(y)\\
&&\quad\le
\int_M V_0(x,\sqrt {st})^{-1}\left(\frac{\sqrt{st}+d(x,y)}{\sqrt{st}}\right)^\Upsilon\exp\left\{-\frac{c(p)\gz d^2(x,y)}{st}\right\}\,d\mu(y)\\
&&\quad\le \int_M V_0(x,\sqrt {st})^{-1}\exp\left\{-\frac{c(p,\gz) d^2(x,y)}{st}\right\}\,d\mu(y)\\
&&\quad\le \sum_{k=1}^\infty V_0(x,\sqrt {st})^{-1} V_0(x,2^k\sqrt {st})\exp\left\{-c(p,\gz)2^{2k}\right\} \\
&&\quad\le C.
\end{eqnarray*}
Inequality \eqref{3.12} gives that
\begin{eqnarray*}
&&\left\|(I-A)\nabla e^{-st\mathcal{L}_0}f\right\|_{p}^p\\
&&\le
C\int_M\int_M |(I-A)|^p|\nabla p^0_{st}(x,y)|^p\exp\left\{\frac{\gz d^2(x,y)}{st}\right\}V_0(y,\sqrt t)^{{p}-1}|f(y)|^p\,d\mu(y) \,d\mu(x).
\end{eqnarray*}
Note that $p<p_0$. Letting $\delta=(p_0-p)/2$, $q=(p+p_0)/2$ and $\gz\in (0,\gz_q)$ such that $2(p+\delta)\gz/p=\gz_q$,  we conclude that
\begin{eqnarray*}
&&\int_M |I-A(x)|^p |\nabla_x p^0_{st}(x,y)|^p\exp\left\{\frac{\gz d^2(x,y)}{st}\right\} \,d\mu(x)\\
&&\le \left(\int_M |I-A(x)|^{p(p+\delta)/\delta}\exp\left\{\frac{-(p+\delta)\gz d^2(x,y)}{\delta st}\right\}\,d\mu(x)\right)^{\delta/(p+\delta)}\\
 &&\quad\times \left(\int_{M}|\nabla_x p^0_{st}(x,y)|^{p+\delta}\exp\left\{\frac{2(p+\delta)\gz d^2(x,y)}{pst}\right\} \,d\mu(x)\right)^{p/(p+\delta)}\\
 &&\le \frac{C}{(st)^{{p}/2}V_0(y,\sqrt {st})^{{p-p/q}}}\left(\int_M |I-A(x)|^{pq/\delta}e^{-q\gz \frac{d(x,y)^2}{\delta st}}\,d\mu(x)\right)^{\delta/q}\\
 &&\le  \frac{C}{(st)^{{p}/2}V_0(y,\sqrt {st})^{{p-1}}}\left(\frac{1}{V_0(y,\sqrt{st})}\int_M |I-A(x)|^{pq/\delta}e^{-q\gz \frac{d(x,y)^2}{\delta st}}\,d\mu(x)\right)^{\delta/q},
\end{eqnarray*}
where by \eqref{pol-decay} one has
\begin{eqnarray*}
&&\left(\frac{1}{V_0(y,\sqrt{st})}\int_M |I-A(x)|^{pq/\delta}e^{-q\gz \frac{d(x,y)^2}{\delta st}}\,d\mu(x)\right)^{\delta/q}\\
&&\le \left(\frac{C}{V_0(y,\sqrt{st})}\int_M |I-A(x)|e^{-q\gz \frac{d(x,y)^2}{\delta st}}\,d\mu(x)\right)^{\delta/q}\\
&&\le  \left(\sum_{k=1}^\infty\frac{e^{-c2^{2k}}}{V_0(y,\sqrt{st})}\int_{B(y,2^k\sqrt{st})} |I-A(x)|\,d\mu(x)\right)^{\delta/q}\\
&&\le C(st)^{-\epsilon\delta/2q}.
\end{eqnarray*}
We can therefore conclude that
\begin{eqnarray*}
&&\int_M |I-A(x)|^p |\nabla_x p^0_{st}(x,y)|^p\exp\left\{\frac{\gz d^2(x,y)}{st}\right\} \,d\mu(x)\le
\frac{C}{(st)^{{p}/2+\epsilon\delta/2q}V_0(y,\sqrt {st})^{{p-1}}},
\end{eqnarray*}
and
\begin{eqnarray*}
&&\|(I-A)\nabla e^{-st\mathcal{L}_0}f\|_{p}^p\\
&&\le
C\int_M\int_M |I-A(x)|^p|\nabla_x p^0_{st}(x,y)|^p\exp\left\{\frac{\gz d^2(x,y)}{st}\right\}V_0(y,\sqrt {st})^{{p}-1}|f(y)|^p\,d\mu(y) \,d\mu(x)\\
&&\le C\frac{C}{(st)^{{p}/2+\epsilon\delta/2q}}\|f\|^p_{p}.
\end{eqnarray*}
We finally get the estimate of the second term by
\begin{eqnarray*}
&&\int_{1/t}^\infty \|(I-A)\nabla e^{-s(1+t\mathcal{L}_0)}f\|_{p}\,ds\le \int_{1/t}^\infty \frac{Ce^{-s}}{(st)^{{1}/2+\epsilon\delta/(2pq)}}\|f\|_{p}\,ds\le Ct^{-1/2-\epsilon\delta/(2pq)}\|f\|_{p}.\\
\end{eqnarray*}
This together with \eqref{3.11} implies that
\begin{eqnarray*}
&&\left\|(I-A)\nabla (1+t\mathcal{L}_0)^{-1}f\right\|_{p}\le Ct^{-1/2-\epsilon\delta/(2pq)}\|f\|_{p}.
\end{eqnarray*}
By the same proof, one sees that
$$\left\|\left(\left|\sqrt{|g_0|/|g|}-1\right|\nabla_0\right)(1+t\mathcal{L}_0)^{-1}\right\|_{p\to p}\le Ct^{-1/2-\epsilon\delta/(2pq)}.$$
The above two estimates complete  the proof of {\bf Step 2}.

Finally, by combining the estimates from {\bf Step 1}  and {\bf Step 2}, we
see that
\begin{eqnarray*}
\|I_t^1\|_{p\to p}\le Ct^{-\epsilon(p_0-p)/(2p(p_0+p))} \|\nabla(1+t\mathcal{L})^{-1/2}\|_{p\to p},
\end{eqnarray*}
which completes the proof.
\end{proof}

In Proposition \ref{metric-lem-2}, there is a  loss of integrability,  which is somehow nature from the point of view of comparing arguments; see also \cite{cp98}. If we strengthen the assumption from $(GUB)$ to
two side bounds of the heat kernel, then by using the open-ended property of
the Riesz transform (cf. \cite{cjks16}) we have the end-point estimate.

We say that the heat kernel satisfies Li-Yau  estimate  if there exist $C,c>0$ such that
for all $t>0$ and all $\,x,y\in M$.
$$ \frac{C^{-1}}{V(x,{\sqrt t})}\exp\lf\{-\frac{d^2(x,y)}{ct}\r\}\le p_t(x,y)\le
  \frac{C}{V(x,{\sqrt t})}\exp\lf\{-c\frac{d^2(x,y)}{t}\r\}.\leqno(LY)$$
By \cite{sal,sal2,gri92}, the Li-Yau estimate is equivalent to $(M,g)$ satisfies $(D)$ and a scale invariant
Poincar\'e inequality $(PI)$, i.e.,
$$
\fint_{B(x,r)}|f-f_B|\,d\mu\le
Cr\lf(\fint_{B(x,r)}|\nabla f|^2\,d\mu\r)^{1/2}.\leqno(PI)$$
As $(D)$ and $(PI)$ are invariant under quasi-isometries, the Li-Yau estimate is invariant under
quasi-isometries.

\begin{prop}\label{metric-lem-3}
Assume that $(M,g_0)$ satisfies $(D)$ and $(PI)$, and that \eqref{quasi-isometry} holds.
Suppose that  $\nabla \mathcal{L}_{0}^{-1/2}$ is bounded on
$L^{p_0}(M)$ for some $p_0\in (2,\infty)$, and there exists $\epsilon>0$ such that
$$\fint_{B_0(x,r)}|\delta_{ik}-g_0^{ij}g_{jk}|\,d\mu_0\le Cr^{-\epsilon},\ \forall\ 1\le i,k\le n,  \ \forall x\in M\,\&\,\forall \,r>1.$$
Then there exist $C>0$ and $\alpha>0$ such that for each $t>1$
\begin{eqnarray*}
\|I_t^1\|_{p_0\to p_0}\le Ct^{-\alpha} \|\nabla(1+t\mathcal{L})^{-1/2}\|_{p_0\to p_0}.
\end{eqnarray*}
\end{prop}
\begin{proof}
Note that, under $(D)$ and $(PI)$, the boundedness of the Riesz transform has an open-ended character, cf. \cite[Theorem 1.9]{cjks16}.
Therefore, there exists $\delta>0$ such that $\nabla \mathcal{L}_{0}^{-1/2}$ is bounded on
$L^{p_0+\delta}(M)$, which together with \cite[Theorem 1.6]{cjks16} implies that
$$\int_M|\nabla_x p^0_t(x,y)|^{p+\delta}\exp\left\{\gz d^2(x,y)/t\right\}\,d\mu_0(x)\le \frac{C}{t^{{(p_0+\delta)}/2}V_0(y,\sqrt t)^{{p_0+\delta}-1}}.\leqno(GLY_{p_0+\delta})$$
Using $(GLY_{p_0+\delta})$ instead of $(GLY_{p})$ in the proof of Proposition \ref{metric-lem-2}, we see that
there exists $\alpha>0$ such that for each $t>1$
\begin{eqnarray*}
\|I_t^1\|_{p_0\to p_0}\le Ct^{-\alpha} \|\nabla(1+t\mathcal{L})^{-1/2}\|_{p_0\to p_0},
\end{eqnarray*}
as desired.
\end{proof}

\begin{thm}\label{metric-perturbation-ly}
Assume that $(M,g_0)$ satisfies $(D)$ and $(PI)$, and that \eqref{quasi-isometry} holds.
Suppose that  $\nabla \mathcal{L}_{0}^{-1/2}$ is bounded on
$L^{p}(M)$ for some $p\in (2,\infty)$, and there exists $\epsilon>0$ such that
$$\fint_{B_0(x,r)}|\delta_{ik}-g_0^{ij}g_{jk}|\,d\mu_0\le Cr^{-\epsilon},\ \forall\ 1\le i,k\le n, \ \forall x\in M\,\&\,\forall \,r>1.$$
Then if $\nabla (1+\mathcal{L})^{-1/2}$  bounded on
$L^p(M)$, $\nabla \mathcal{L}^{-1/2}$ is bounded on $L^p(M)$.
\end{thm}
\begin{proof}
The conclusion follows from the same proof of Theorem \ref{main-manifold}, using Proposition \ref{metric-lem-3} instead of Proposition \ref{metric-lem-2}.
\end{proof}

We can now finish the proofs for corollaries of Theorem \ref{main-manifold}.
\begin{proof}[Proof of Corollary \ref{cor-lowerRicci}]
 Noting that $(M,g)$ has Ricci curvature bounded from below, the local Riesz transform $\nabla (1+\L)^{-1/2}$
 is bounded on $L^p(M)$ for all $p\in (1,\infty)$; see \cite{acdh}. The conclusion then follows from Theorem \ref{main-manifold}.
\end{proof}

%\begin{cor}
%Assume that $g,\,g_0$ are two metrics on $M$, that satisify  \eqref{quasi-isometry} and there exists $\epsilon>0$ such that
%$$\fint_{B_r(x)}|\delta_{ik}-a_{ik}|\,d\mu\le Cr^{-\epsilon},\ \forall x\in M\,\&\,\forall \,r>1.$$
%Suppose that $(M,g)$ has Ricci curvature bounded from below and satisfies $(D)$ and $(PI)$.
%Then if  $\nabla \mathcal{L}_{0}^{-1/2}$ is bounded on $L^{p}(M)$ for some $p\in (2,\infty)$,
%$\nabla \mathcal{L}^{-1/2}$ is bounded on $L^p(M)$.
%\end{cor}

\begin{proof}[Proof of Corollary \ref{cor-smooth}]
If $g$ coincides with  $g_0$ outside a compact subset $M_0$, then \eqref{quasi-isometry} holds.
By using $(D)$ together with the connectivity of $M$, one sees there exists $0<\upsilon\le \Upsilon<\infty$  such that for any $y\in M$ and $0<r<R<\infty$
\begin{equation}\label{reverse-doubling}
\frac{1}{C}\left(\frac{R}{r}\right)^{\upsilon}\le  \frac{V_0(y,R)}{V_0(y,r)}\le C \left(\frac{R}{r}\right)^{\Upsilon};
\end{equation}
see \cite{hkst} for instance.

Note that it holds
$$\fint_{B_0(x,r)}|\delta_{ik}-g_0^{ij}g_{jk}|\,d\mu_0\le \frac{C}{V_0(x,r)}\int_{M_0\cap B_0(x,r)}|\delta_{ik}-g_0^{ij}g_{jk}|\,d\mu_0,\ \forall x\in M\,\&\,\forall \,r>1.$$
Fix $x_0\in M_0$. If $B_0(x,r)\cap M_0=\emptyset$, then
\begin{equation}\label{case-empty}
\fint_{B_0(x,r)}|\delta_{ik}-g_0^{ij}g_{jk}|\,d\mu_0=0.
\end{equation}
Otherwise, by using \eqref{reverse-doubling}, one has
\begin{eqnarray*}
V_0(x_0,\mathrm{diam}(M_0))&&\le C\left(\frac{\mathrm{diam}(M_0)}{r+\mathrm{diam}(M_0)}\right)^\upsilon V_0(x_0,r+\mathrm{diam}(M_0))\\
&&\le C\left(\frac{\mathrm{diam}(M_0)}{r+\mathrm{diam}(M_0)}\right)^\upsilon V_0(x,2r+2\mathrm{diam}(M_0))\\
&&\le C\left(\frac{\mathrm{diam}(M_0)}{r+\mathrm{diam}(M_0)}\right)^\upsilon
\left(\frac{r+\mathrm{diam}(M_0)}{r}\right)^\Upsilon V_0(x,r),\\
\end{eqnarray*}
and therefore,
\begin{eqnarray*}
&&\fint_{B_0(x,r)}|\delta_{ik}-g_0^{ij}g_{jk}|\,d\mu_0\\
&&\le C\left(\frac{\mathrm{diam}(M_0)}{r+\mathrm{diam}(M_0)}\right)^\upsilon
\left(\frac{r+\mathrm{diam}(M_0)}{r}\right)^\Upsilon \frac{1}{V_0(x_0,\mathrm{diam}(M_0))}\int_{M_0} |\delta_{ik}-g_0^{ij}g_{jk}|\,d\mu_0.
\end{eqnarray*}
This together with \eqref{case-empty} implies that for all $x\in M$ and all $r>1$ it holds
$$\fint_{B_0(x,r)}|\delta_{ik}-g_0^{ij}g_{jk}|\,d\mu_0\le \frac{C}{r^\upsilon},$$
where $C$ depends on $M_0$.

By applying Theorem \ref{main-manifold}, together with that $(M,g)$ and $(M,g_0)$ having lower Ricci curvature bounds, we see $\nabla_0\L_0^{-1/2}$ is bounded on $L^p(M)$ for all $p\in (2,p_0)$,
if and only if, $\nabla\L^{-1/2}$ is bounded on $L^p(M)$ for all $p\in (2,p_0)$.
\end{proof}

\begin{proof}[Proof of Corollary \ref{cor-nonnegative}]
Note that since $(M,g_0)$ has non-negative Ricci curvature, $\nabla_0\L_0^{-1/2}$
is bounded on $L^p(M)$ for all $p\in (1,\infty)$; see \cite{acdh}. This together with Corollary \ref{cor-smooth} implies that
$\nabla \L^{-1/2}$ is bounded on $L^p(M)$ for all $p\in (2,\infty)$.

Moreover, since $(D)$ and $(GUB)$ hold on $(M,g_0)$ as a consequence of non-negative Ricci curvature (cf. \cite{ly86}),
$(D)$ and $(GUB)$ hold on $(M,g)$. By \cite{cd99} we see that $\nabla \L^{-1/2}$ is bounded on $L^p(M)$ for all $p\in (1,2)$.
\end{proof}

\section{Degenerate elliptic equations}\label{sec-degenerate}
\hskip\parindent In this section, we deal with degenerate elliptic equations on Euclidean spaces.
Let $A_2(\rn)$ denote the collection of $A_2$-Muckenhoupt weights, and $QC(\rn)$ denote the collection of all quasi-conformal weights, i.e.,
$w\in QC(\rn)$ if there exists a quasi-conformal mapping $f:\,\rn\to\rn$ such that $w=|J_f|^{1-2/n}$, where $J_f$ denotes the determinant
of the gradient matrix $Df$; see \cite{FKS82}.

For $w,w_0\in A_2(\rn)\cup QC(\rn)$, denote by
$$V(x,r)=\int_{B(x,r)}w\,dy, \ \ \ V_0(x,r)=\int_{B(x,r)}w_0\,dy,\ \, \forall \,x\in\rn \&\, r>0.$$
For $w,w_0\in A_2(\rn)\cup QC(\rn)$, the volumes $V,V_0$ satisfy the doubling condition,
and there are scale-invariant Poincar\'e inequality $(PI)$ on the spaces $(\rn,w\,dx)$ and $(\rn,w_0\,dx)$;  see \cite{FKS82}.

We will assume that $C^{-1}w\le w_0\le Cw$ in what follows. As a consequence of the assumption,
it holds that $C^{-1}V(x,r)\le V_0(x,r)\le CV(x,r)$ for any $x\in\rn$ and $r>0$, and $L^p(w)=L^p(w_0)$ for any $p>0$.
In what follows, we will not distinguish $L^p(w)$ and $L^p(w_0)$, and denote by $\|\cdot\|_{p\to p}$
the operator norm $\|\cdot\|_{L^p(w)\to L^p(w)}$.

In this section, we use $\nabla,\,\mathrm{div}$ to denote the gradient operator and divergence operator on $\rn$. We use the notation $B(x,r)$ for open ball under usual Euclidean metric of $\rn$.
The proof of results in this section is similar to that of Theorem \ref{main-manifold}, thus we only sketch their proofs.

\begin{thm}\label{main-denegerate}
Let $A,A_0$ be $n\times n$ matrixes that satisfy  uniformly elliptic  conditions, and
$w_0,w\in A_2(\rn)\cup QC(\rn)$ with
$$C^{-1}w(x)\le w_0(x)\le Cw(x),\,a.e.\, x\in \rn.$$
Suppose there exists $\epsilon>0$ such that
$$\frac{1}{V_0(y,r)}\int_{B(y,r)} \left(|A-A_0|+\frac{|w_0-w|}{w_0}\right) w_0\,dx\le \frac{C}{r^\epsilon},\quad\, \forall \,y\in\rn \, \& \,  r>1.$$
Let $\L=-\frac{1}{w}\mathrm{div}(wA\nabla)$ and $\L_0=-\frac{1}{w_0}\mathrm{div}(w_0A_0\nabla).$
Then if $\nabla\mathcal L_0^{-1/2}$ and $\nabla(1+\mathcal L)^{-1/2}$  are bounded on $L^p(w)$ for some $p\in (2,\infty)$,
$\nabla\mathcal L^{-1/2}$ is bounded on $L^p(w)$.
\end{thm}
\begin{proof}
Let us begin with the formula that for any $t>0$
\begin{eqnarray*}
\nabla(1+t\L_0)^{-1}-\nabla(1+t\L)^{-1}&&=t\nabla(1+t\L)^{-1}(\L_0-\L)(1+t\L_0)^{-1},
\end{eqnarray*}
and
\begin{eqnarray*}
\nabla \mathcal{L}_0^{-1/2}-\nabla \mathcal{L}^{-1/2}&&=\frac{1}{\pi}\int_0^1 \nabla[(1+t\mathcal{L}_0)^{-1}-(1+t\mathcal{L})^{-1}]\frac{\,dt}{\sqrt t}\\
&&+\frac{1}{\pi}\int_1^\infty\nabla[(1+t\mathcal{L}_0)^{-1}-(1+t\mathcal{L})^{-1}]\frac{\,dt}{\sqrt t}.
\end{eqnarray*}

{\bf Step 1}. By Lemma \ref{lemma-s1} and our assumptions that
$\nabla\mathcal L_0^{-1/2}$ and $\nabla(1+\mathcal L)^{-1/2}$  are bounded on $L^p(w_0)$, we have
$$\left\|\int_0^1 \nabla[(1+t\mathcal{L}_0)^{-1}-(1+t\mathcal{L})^{-1}]\frac{\,dt}{\sqrt t}\right\|_{p\to p}\le C.$$

{\bf Step 2.} For the remaining term, by the following Proposition \ref{degene-lem-1} and Proposition \ref{degene-lem-2},
we see that there exists $\alpha>0$ such that for any $t>1$
\begin{equation}\label{degene-2}
\|\nabla[(1+t\mathcal{L}_0)^{-1}-(1+t\mathcal{L})^{-1}]\|_{p\to p}\le Ct^{-\alpha}\|\nabla(1+t\mathcal{L})^{-1/2}\|_{p\to p}.
\end{equation}
With this, repeating the same proof after \eqref{metric-2} in the proof of Theorem \ref{main-manifold},
we conclude that
$\nabla \mathcal{L}^{-1/2}$ is bounded on $L^p(w)$. %
\end{proof}

Let us prove \eqref{degene-2} in the following two propositions.  Note that
\begin{eqnarray*}
\L_0-\L&&=\frac{1}{w}\mathrm{div}\left[((w-w_0)A-w_0(A_0-A))\nabla\right]-\frac{w-w_0}{w_0w}\mathrm{div}(w_0A_0\nabla),
\end{eqnarray*}
and set
\begin{eqnarray*}
I_t^1=&& t\nabla(1+t\mathcal{L})^{-1}\left(\frac{1}{w}\mathrm{div}\left(\left[(w-{w_0})A-{w_0}(A_0-A)\right]\nabla (1+t\mathcal{L}_0)^{-1}\right)\right),
\end{eqnarray*}
and
\begin{eqnarray*}
II_t^2=t\nabla(1+t\mathcal{L})^{-1}\left(\frac{w-w_0}{w}\L_0(1+t\mathcal{L}_0)^{-1}\right).
\end{eqnarray*}
Then we have the following representation
\begin{eqnarray*}
\int_1^\infty \nabla[(1+t\mathcal{L}_0)^{-1}-(1+t\mathcal{L})^{-1}]\frac{\,dt}{\sqrt t}=\int_1^\infty(I_t^1+II_t^2)\frac{\,dt}{\sqrt t}.
\end{eqnarray*}

\begin{prop}\label{degene-lem-1}
Let $A,A_0$ be $n\times n$ matrixes that satisfy  uniformly elliptic  conditions, and
$w_0,w\in A_2(\rn)\cup QC(\rn)$ with
$$C^{-1}w(x)\le w_0(x)\le Cw(x),\,a.e.\, x\in \rn.$$
Suppose there exists $\epsilon>0$ such that
$$\frac{1}{V_0(y,r)}\int_{B(y,r)} \left(\frac{|w_0-w|}{w_0}\right) w_0\,dx\le \frac{C}{r^\epsilon},\quad\, \forall \,y\in\rn \, \& \,  r>1.$$
Then for each $p\in (2,\infty)$, there exists $C>0$ such that for each $t>1$ it holds
$$\|II_t^2\|_{p\to p}\le Ct^{-\epsilon(p-1)/2p}\|\nabla(1+t\mathcal{L})^{-1/2}\|_{p\to p} .$$
\end{prop}
\begin{proof}
Note that $(D)$ and $(PI)$ hold since $w_0,w\in A_2(\rn)\cup QC(\rn)$, the heat kernels $p^\L_t(x,y)$, $p^{\L_0}_t(x,y)$
of $e^{-t\L}$, $e^{-t\L_0}$ satisfy $(LY)$ estimates by \cite{sal,sal2}, i.e.,
$$p_{t}^\L(x,y),\,p_{t}^{\L_0}(x,y) \sim \frac{1}{V(x,\sqrt{t})}e^{-\frac{d(x,y)^2}{ct}}.$$
Using this, and the assumption that
$$\frac{1}{V_0(y,r)}\int_{B(y,r)} \frac{|w_0-w|}{w_0} w_0\,dx\le \frac{C}{r^\epsilon},\quad\, \forall \,y\in\rn \, \& \,  r>1,$$
we follow the proof of Proposition \ref{metric-lem-1} to see that
$$\left\|(1+t\mathcal{L})^{-1/2}\frac{w-w_0}{w}\right\|_{p\to p}\le Ct^{-\epsilon(p-1)/2p}.$$
This, together with the estimate
$$\|\L_0(1+t\L_0)^{-1}\|_{p\to p}\le C/t$$
for all $t>1$, implies
\begin{eqnarray*}
\|II_t^2\|_{p\to p}\le Ct^{-\epsilon(p-1)/2p}\|\nabla(1+t\mathcal{L})^{-1/2}\|_{p\to p}.
\end{eqnarray*}
\end{proof}

\begin{prop}\label{degene-lem-2}
Let $A,A_0$ be $n\times n$ matrixes that satisfy  uniformly elliptic  conditions, and
$w_0,w\in A_2(\rn)\cup QC(\rn)$ with
$$C^{-1}w(x)\le w_0(x)\le Cw(x),\,a.e.\, x\in \rn.$$
Suppose there exists $\epsilon>0$ such that
$$\frac{1}{V_0(y,r)}\int_{B(y,r)} \left(|A-A_0|+\frac{|w_0-w|}{w_0}\right) w_0\,dx\le \frac{C}{r^\epsilon},\quad\, \forall \,y\in\rn \, \& \,  r>1.$$
Let $\L=-\frac{1}{w}\mathrm{div}(wA\nabla)$ and $\L_0=-\frac{1}{w_0}\mathrm{div}(w_0A_0\nabla).$
Then if $\nabla\mathcal L_0^{-1/2}$ is bounded $L^p(w)$ for some $p\in (2,\infty)$, there exist $C,\,\alpha>0$ such that for any $t>1$
\begin{eqnarray*}
\|I_t^1\|_{p\to p}\le Ct^{-\alpha} \|\nabla(1+t\mathcal{L})^{-1/2}\|_{p\to p}.
\end{eqnarray*}
\end{prop}
\begin{proof} Recall that
\begin{eqnarray*}
I_t^1=&& t\nabla(1+t\mathcal{L})^{-1}\left(\frac{1}{w}\mathrm{div}w\left[\left(\left(1-\frac{w_0}{w}\right)A-\frac{w_0}{w}(A_0-A)\right)\nabla (1+t\mathcal{L}_0)^{-1}\right]\right),
\end{eqnarray*}

{\bf Step 1}.
Since $(1+t\mathcal{L})^{-1/2}\frac{1}{w}\mathrm{div}w$ is the dual operator of $\nabla (1+t\mathcal{L})^{-1/2}$,
and $\nabla \mathcal{L}^{-1/2}$  is bounded on $L^{p'}(w)$,
one has that
$$ \left\|(1+t\mathcal{L})^{-1/2}\frac{1}{w}\mathrm{div}w\right\|_{p\to p}=
\left\|\nabla (1+t\mathcal{L})^{-1/2}\right\|_{{p'}\to p'}\le C/\sqrt t.$$

{\bf Step 2}.
As $w_0,w\in A_2(\rn)\cup QC(\rn)$, $(D)$ and $(PI)$ hold. By \cite[Theorem 1.6 \& Theorem 1.9]{cjks16},
we see that there exists $p_0>p$ such that $\nabla\mathcal L_0^{-1/2}$ is bounded on $L^{p_0}(w)$ and
$$\int_\rn|\nabla_x p^{\L_0}_t(x,y)|^{p_0}\exp\left\{\gz d^2(x,y)/t\right\}w_0(x)\,dx\le \frac{C}{t^{{p_0}/2}V_0(y,\sqrt t)^{{p_0}-1}},\leqno(GLY_{p_0})$$
Using this together with the assumption
$$\frac{1}{V_0(y,r)}\int_{B(y,r)} \left(|A-A_0|+\frac{|w_0-w|}{w_0}\right) w_0\,dx\le \frac{C}{r^\epsilon},\quad\, \forall \,y\in\rn \, \& \,  r>1,$$
we follow the proof of Proposition \ref{metric-lem-2} and Proposition \ref{metric-lem-3} to conclude that there exists $\alpha>0$ such that
$$\left\|\left(\left(1-\frac{w_0}{w}\right)A-\frac{w_0}{w}(A_0-A)\right)\nabla (1+t\mathcal{L}_0)^{-1}\right\|_{p\to p}\le Ct^{-1/2-\alpha}.$$
The above two steps give the desired estimates.
\end{proof}

Finally by using Theorem \ref{main-denegerate} and the result of \cite{cp98} we can finish the proof of Corollary \ref{cor-euclidean}.
Note that $L^p$-boundedness of the Riesz transform on $L^p(\rn)$ for $p\in (1,2)$ is always true if
the operator is uniformly elliptic; see \cite{cd99}.

\begin{proof}[Proof of Corollary \ref{cor-euclidean}]
By \cite{cp98}, if the matrix $A$ is uniformly continuous on $\rn$, then
every solution to $\L u=0$ on $B(x,r)$, $r<1$, satisfies
$$\left(\fint_{B(x,r/2)}|\nabla u|^q\,dy\right)^{1/q}\le \frac{C}{r}\fint_{B(x,r)}|u|\,dy$$
for any $q<\infty$. This implies the local Riesz operator $\nabla (1+\L)^{-1/2}$ is $L^q$-bounded
for any $q<\infty$; see \cite{acdh,cjks18}. The same holds for $\L_0$.
The conclusion then follows from Theorem \ref{main-denegerate}.
\end{proof}
For the homogenized elliptic operator $\L_0=-\mathrm{div}A\nabla$ (cf. \cite{AL91}), it was known by \cite{AL91} that
$\nabla \L_0^{-1/2}$ is bounded on $L^p(\rn)$ for all $p\in (1,\infty)$. Therefore, if
$\L=-\mathrm{div}B\nabla$ with $B$ uniformly continuous and satisfying
$$\fint_{B(y,r)} \left|A-B\right|\,dx\le \frac{C}{r^\epsilon},\quad\, \forall \,y\in\rn \, \& \, \forall \, r>1,$$
for some $\epsilon>0$, then Corollary \ref{cor-euclidean} implies that $\nabla \L^{-1/2}$ is bounded on $L^p(\rn)$ for all $p\in (1,\infty)$.

\section{Examples}\label{example}
\hskip\parindent In this section, we discuss the couter-examples regarding unboundedness of the Riesz operator.
\subsection{Conic Laplace operator}
\hskip\parindent Let us start from the Meyer's example; see \cite{at98}, and also \cite{lin96} for general asymptotically conic elliptic operators.

{\bf On the plane.} Consider $\L=-\mathrm{div} A\nabla$, where
$$A(x)=I+\frac{\beta(\beta+2)}{|x|^2}\left(\begin{array}{cc}
x_2^2& \ -x_1x_2\\
-x_1x_2&\ x_1^2\end{array}\right),$$
where $\beta\in (-1,\infty)$. Then $f(x)=|x|^{\beta}x_1$ is a (weak-)solution to $\L f=0$ on $\rr^2$.
In particular, for $\beta\in (-1,0)$, $\nabla f(x)$ is not locally $L^p$ integrable around the origin for any $p\ge 2/|\beta|$.
This implies the Riesz operator $\nabla \L^{-1/2}$ can not be bounded on $L^p(\rr^2)$ for any $p\ge 2/|\beta|$; see \cite{shz05,cjks16}.

To see the geometric meaning of $\L$, let us rewrite $\L$ in the polar coordinates.
First let $\lambda=\beta(\beta+2)+1$, and write
$$A(x)=I+\frac{\beta(\beta+2)}{|x|^2}\left(\begin{array}{cc}
x_2^2& \ -x_1x_2\\
-x_1x_2&\ x_1^2\end{array}\right)=\frac{1}{|x|^2}\left(\begin{array}{cc}
x_1^2& \ x_1x_2\\
x_1x_2&\ x_2^2\end{array}\right)+\lambda I+\frac{\lambda}{|x|^2}\left(\begin{array}{cc}
-x_1^2& \ -x_1x_2\\
-x_1x_2&\ -x_2^2\end{array}\right).$$
Set
$$A_1(x)=\frac{1}{|x|^2}\left(\begin{array}{cc}
x_1^2& \ x_1x_2\\
x_1x_2&\ x_2^2\end{array}\right),$$
and  $A_2(x)=A(x)-A_1(x)$. Then in the polar coordinates, the operator $\L_1=-\mathrm{div} A_1\nabla$ has the representation
$$\L_1 f=-\frac 1r \frac{\partial }{\partial r} \left(r\frac{\partial f}{\partial r}\right),$$
and $\L_2=-\mathrm{div} A_2\nabla$ can be represented as
$$\L_2 f=-\frac{\lambda}{r^2}\frac{\partial^2 f }{\partial \theta^2}.$$
This means
$$\L f=-\frac 1r \frac{\partial }{\partial r} \left(r\frac{\partial f}{\partial r}\right)-\frac{\lambda}{r^2}\frac{\partial^2 f }{\partial \theta^2}.$$
It is easy to see that  $f(x)=f(r,\theta)=r^{1+\beta}\cos\theta$ satisfies $\L f=0$ on $\rr^2$ (in the weak sense).
%
%In a geometric point of view, the operator $\L$ was obtained by   dilating the metric of the unit sphere  by a fact $1/\lambda$, $\lambda>0$.

\

{\bf On $\rr^N$, $N\ge 3$}. It is straight to generalize the above operator to higher dimension as
$$\L f=-\frac{1}{r^{N-1}}\frac{\partial }{\partial r} \left(r^{N-1}\frac{\partial f}{\partial r}\right)-\frac{\lambda}{r^2}\Delta_{\mathbb{S}^{N-1}}f,$$
where $\Delta_{\mathbb{S}^{N-1}}$ is the spherical Laplacian operator, and $\lambda>0$.
In the Euclidean coordinates, the operator has the form $\L f=-\mathrm{div} A\nabla f$, where
$$A(x)=\frac{1}{|x|^2}A_{N}+\lambda I-\frac{\lambda}{|x|^2}A_N,$$
where $A_N=\{x_ix_j\}_{1\le i,j\le N}$ and $\lambda\in (0,\infty)$.

Then functions $f_i(x)=|x|^{\beta}x_i$, where $\beta>-1$ satisfying
$$\beta=\sqrt{\frac{N^2}{4}+\lambda(N-1)-N+1 }-\frac N2=\sqrt{\left(\frac {N}{2}-1\right)^2+\lambda(N-1)},$$
$1\le i\le N$,
satisfy  $\L f_i=0$ on $\rr^N$.

In particular, if $\lambda\in (0,1)$, then $\beta\in (-1,0)$ and the gradient $\nabla f_i$
does not belong to $L^p_\loc(\rr^N)$ for any
$$p\ge \frac{N}{|\beta|}=N\left(\frac N2-\sqrt{\left(\frac {N}{2}-1\right)^2+\lambda(N-1)}\right)^{-1}.$$
This implies that the Riesz transform $\nabla \L^{-1/2}$ can not be bounded on  $L^p(\rr^N)$ for $p\ge N/|\beta|$;
see \cite{shz05,cjks16}.

Viewing $\lambda(N-1)$ as the lowest non-zero eigenvalue of the operator
$\lambda\Delta_{\mathbb{S}^{N-1}}$, this range coincides with Lin \cite{lin96}
of the conical elliptic operators, and also Li \cite{lh99} on general conic manifolds.

As $\beta\in (-1,0)$, the above generalization of conic Laplacian operator to higher dimensions
gives counter-example of failure of the boundedness of the Riesz transform for any $p>N$.

\

For the counter-example regarding the case $p\in (2,N]$, let us consider the operator $\L=-\mathrm{div}A\nabla $ given by the matrix
$$A(x)=I+\frac{\beta(\beta+2)}{x_1^2+x_2^2}\left(\begin{array}{ccccc}
x_2^2& \ -x_1x_2 & \ 0 & \cdots & \ 0\\
-x_1x_2&\ x_1^2  & \ 0 & \cdots & \ 0\\
0& \ 0& \ 0 & \cdots & \ 0\\
\cdots\\
0& \ 0& \ 0 & \cdots & \ 0
\end{array}\right),$$
where $\beta\in (-1,0)$. The functions $f_1(x)=(x_1^2+x_2^2)^{\beta/2}x_1$, $f_2(x)=(x_1^2+x_2^2)^{\beta/2}x_2$
are weak  solutions to $\L u=0$ on $\rr^N$. However, the gradients $\nabla f_1,\,\nabla f_2$
do not belong to $L^p_\loc(\rr^N)$ for any
$$p\ge \frac{2}{|\beta|}.$$
This shows the corresponding Riesz transform can not bounded on $L^p(\rr^N)$ for any $p\ge 2/|\beta|$.

\subsection{Uniformly elliptic operator with smooth coefficients}
\hskip\parindent  The previous examples shows that on $\rr^N$, $N\ge 2$, for any $p>2$, there are uniformly elliptic operators
such that the Riesz transform is not $L^p$-bounded.
However, these operators are not smooth at the origin.
We next provide a uniformly elliptic operator with smooth coefficients such that the Riesz transform
is not $L^p$-bounded, for any given $p>2$. The idea of the construction comes
  from the study of asymptotically conic elliptic operators in \cite{lin96} and the homogenization theory
  (cf. \cite{AL87,AL87b}).

\begin{thm}\label{matrix-smooth}
For any given matrix $A(x)$ with measurable coefficients satisfying uniformly elliptic conditions,
there exists a smooth matrix $B(x)$ satisfying uniformly elliptic conditions, whose derivatives of any order
are bounded, and a sequence $\{r_k\}_{k\in\cn}$, $r_k\to \infty$, such that $B(r_k x)\to A(x)$ a.e. on $\rn$.
\end{thm}
\begin{proof}
Let $0\le \psi\in C^\infty(\rn)$ be a mollifier, i.e. $\supp \psi\subset B(0,1)$ and $\int_\rn \psi\,dx=1$.

Choose an increasing sequence $\{r_k\}_{k\in\cn}$, $1<r_k\to \infty$, such that
$$r_k <r_k^2<2r_k^2<\sqrt {r_{k+1}}-1$$
for any $k\in\cn$.

Set
$$\tilde B(x):=\left\{\begin{array}{cc}
A(x/r_k), & \ \mbox{if} \ x\in \bigcup_{k\in\cn} B(0,2r_k^2)\setminus B(0,\sqrt{r_k}-1),\\
I_{n\times n}, & \mbox{other}\ x.
\end{array}
\right.
$$
Then $\tilde B$ satisfies the uniformly elliptic conditions. Let
$$B(x):=\tilde B\ast \psi(x). $$
Then $B$ is a $C^\infty$ matrix satisfying uniform ellipticity and any order of its gradients is bounded.

Note that for $x\in B(0,r_k)\setminus B(0,1/\sqrt{r_k})$, for all $m>k$ it holds that
$$r_mx\in B(0,r_kr_m)\setminus B(r_m/\sqrt{r_k})\subset B(0,r_m^2)\setminus B(\sqrt{r_m}).$$
By the choose of $B$, we see that it holds for all $m>k$ that
\begin{eqnarray*}
B(r_mx)=\int_\rn A(y/r_m)\psi(r_mx-y)\,dy=r_m^n\int_\rn A(y)\psi(r_m(x-y))\,dy.
\end{eqnarray*}
Letting $m\to \infty$ yields that for a.e. $x\in B(0,r_k)\setminus B(0,1/\sqrt{r_k})$ it holds
\begin{eqnarray*}
\lim_{m\to\infty}B(r_mx)=A(x).
\end{eqnarray*}
Therefore, we see that for a.e. $x\in\rn$, it holds
\begin{eqnarray*}
\lim_{m\to\infty}B(r_mx)=A(x).
\end{eqnarray*}
\end{proof}
We can now finish the proof of Proposition \ref{counter-example}.
\begin{proof}[Proof of Proposition \ref{counter-example}]
From previous subsection, for any $p>2$ on $\rr^n$, $n\ge 2$, we may find a uniformly elliptic operator $\L=-\mathrm{div}A\nabla$
such that there exists a solution $u$ to $\L u=0$ on $\rr^n$ and $\nabla u$ is not $L^p$ integrable near the origin.

By Theorem \ref{matrix-smooth}, we may find a smooth matrix $B(x)$ satisfying the uniform ellipticity and a positive sequence
$\{r_k\}_{k\in\cn}$, $r_k\to \infty$, such that $B(r_k x)\to A(x)$ a.e. on $\rn$.

We claim that for $\L_0=-\mathrm{div}B\nabla $ the Riesz transform $\nabla \L_0^{-1/2}$ is not bounded on $L^p(\rr^n)$.

Let us argue by contradiction. Assume that $\nabla \L_0^{-1/2}$ is bounded on $L^p(\rr^n)$. Then by \cite{shz05,cjks16},
one sees that there exists $C>0$ such that for any ball $B(x,2r)$ and any $\L_0$-harmonic function $v$ on $B(x,2r)$, it holds
$$\left(\fint_{B(x,r)}|\nabla v|^p\,dy\right)^{1/p}\le \frac{C}{r}\fint_{B(x,2r)}|v|\,dy.\leqno(RH_p)$$
For any $k\in\cn$ consider
$$\left\{\begin{array}{cc}
-\mathrm{div}B(r_k\cdot)\nabla v_k=0, &  \mbox{on}\ B(0,1),\\
v_k=u, & \mbox{on}\ \partial B(0,1).
\end{array}
\right.
$$
Then there exists a subsequence, still denoted by $\{v_k\}_{k\in\cn}$, such that $v_k$ converges weakly
to $\tilde u$ in $W^{1,2}(B(0,1))$. Moreover, by using the $G$-convergence (cf. \cite{blp11}), there exists a limit operator $\tilde \L$, such that $\tilde \L$ is a uniformly elliptic operator
 and $\tilde \L\tilde u=0$. By the construction of $B$, one can infer that
$\tilde \L=\L$ and $\tilde u=u$. Moreover, the boundary condition implies that
 $$\|v_k\|_{W^{1,2}(B(0,1))}\le C\|u\|_{W^{1,2}(B(0,1))}.$$

Note that $v_k(x/r_k)\in W^{1,2}(B(0,r_k))$ satisfies $-\mathrm{div}B\nabla v_k(\cdot/r_k)=0$.
By using $(RH_p)$, one has
\begin{eqnarray*}\left(\fint_{B(0,r_k/2)}r_k^{-p}|\nabla v_k(y/r_k)|^p\,dy\right)^{1/p}\le \frac{C}{r_k}\fint_{B(0,r_k)}|v_k(y/r_k)|\,dy
\le \frac{C}{r_k}\fint_{B(0,1)}|v_k(x)|\,dx,
\end{eqnarray*}
and hence,
\begin{eqnarray*}\left(\fint_{B(0,1/2)}|\nabla v_k(x)|^p\,dx\right)^{1/p}\le C\fint_{B(0,1)}|v_k(x)|\,dx\le C\|u\|_{W^{1,2}(B(0,1))}.
\end{eqnarray*}
By applying Poincar\'e inequality, one has $v_k\in L^p(B(0,1/2))$ with
$$\|v_k\|_{L^p(B(0,1/2))}\le C\|u\|_{W^{1,2}(B(0,1))}.$$
These imply that $\{v_k\}_{k\in\cn}$ is a bounded sequence in $W^{1,p}(B(0,1/2))$, and there exists a subsequence $\{v_{k_j}\}_{j\in\cn}$
such that $v_{k_j}$ converges weakly to $u_0$ in $W^{1,p}(B(0,1/2))$.

This further implies $u=u_0\in W^{1,p}(B(0,1/2))$. This contradicts with the fact $\nabla u$ is not $L^p$-integrable around the origin.
Therefore, our claim holds, i.e., the Riesz transform $\nabla \L_0^{-1/2}$ is not bounded on $L^p(\rr^n)$,
where $\L_0=-\mathrm{div}B\nabla $.
\end{proof}

\appendix

\section{Appendix}\label{appendix}
\hskip\parindent In this appendix, we provide some basic fact of functional calculus.
These estimates are more or less well-known (see the proofs in \cite{cdn07} for instance), we include
them for completeness. Note that $(D)$ or $(GUB)$ is not required in this appendix.
Let $X$ be a locally compact, separable, metrisable, and connected
space equipped with a  Borel measure $\mu$ that is finite on compact sets and strictly positive on non-empty open sets.
Consider a strongly local and regular Dirichlet form $\E$ with the domain $\D$ on $L^2(X,\mu)$; see \cite{fot,cjks16} for precise definitions. Such a form can be written as
$$\E (f, g)= \int_X \,d\Gamma(f,g)=\int_X \langle\nabla f,\nabla g\rangle\,d\mu$$
for all $f, g\in \D$.

Corresponding to such a Dirichlet form $\E$, there exists an operator  denoted by $\mathcal{L} $,
acting on a dense domain $\mathscr{D}(\mathcal{L})$ in
$L^2(X,\mu)$, $\mathscr{D}(\mathcal{L})\subset \D$, such that for all
$f\in \mathscr{D}(\mathcal{L})$ and each
$g\in \D$,
$$\int_X f(x)\mathcal{L} g(x)\,d\mu(x)=\mathscr{E}(f,g).$$
The heat semigroup further satisfies  for any $q\in [1,\infty]$ and $t>0$ that
$$\|e^{-t\mathcal{L}}\|_{q\to q}\le 1,$$
where we denote the operator norm $\|\cdot\|_{L^p(X,\mu)\to L^p(X,\mu)}$ by $\|\cdot\|_{p\to p}$.
This implies for for any $s,t>0$ that
\begin{eqnarray}\label{a1}
\|(s+t\mathcal{L})^{-1}\|_{q\to q}&&\le \int_0^\infty \|e^{-r(s+t\mathcal{L})}\|_{q\to q}\,dr
 \le \int_0^\infty e^{-rs}\,dr\le \frac{1}{s}.
\end{eqnarray}
and
\begin{eqnarray}\label{a2}
\|(s+t\mathcal{L})^{-1/2}\|_{q\to q}&&\le C\int_0^\infty \|e^{-r(s+t\mathcal{L})}\|_{q\to q}\frac{\,dr}{\sqrt r}
 \le C\int_0^\infty e^{-rs}\frac{\,dr}{\sqrt r}\le \frac{C}{\sqrt{s}}.
\end{eqnarray}

\begin{lem}\label{lemma-s1}
Let $(X,\mu,\E)$ be a Dirichlet metric measure space. If the local Riesz transform
$|\nabla (1+\mathcal{L})^{-1/2}|$ is bounded on $L^p(X,\mu)$ for some $p\in (2,\infty)$, then the operator
$$\int_0^1|\nabla (1+t\mathcal{L})^{-1}|\frac{\,dt}{\sqrt t}$$
is bounded on $L^p(X,\mu)$.
\end{lem}
\begin{proof}
Note that
\begin{eqnarray*}
\pi(1+\mathcal{L})^{-1/2}-\int_0^1  (1+t\mathcal{L})^{-1}\frac{\,dt}{\sqrt t}&&=
\int_0^1 [(1+t+t\mathcal{L})^{-1}-(1+t\mathcal{L})^{-1}]\frac{\,dt}{\sqrt t}+\int_1^\infty(1+t+t\mathcal{L})^{-1}\frac{\,dt}{\sqrt t}\\
&&=-\int_0^1 -t(1+t+t\mathcal{L})^{-1}(1+t\mathcal{L})^{-1}\frac{\,dt}{\sqrt t}+\int_1^\infty(1+t+t\mathcal{L})^{-1}\frac{\,dt}{\sqrt t}.
\end{eqnarray*}
From this and using \eqref{a1}, \eqref{a2} and the condition that $|\nabla (1+\mathcal{L})^{-1/2}|$ is bounded on $L^p(X,\mu)$, we  deduce that
\begin{eqnarray*}
&&\int_0^1 \left\|-t|\nabla (1+t+t\mathcal{L})^{-1}|(1+t\mathcal{L})^{-1}\right\|_{p\to p}\frac{\,dt}{\sqrt t}\\
&&\le
\int_0^1 \left\|-t|\nabla (1+\mathcal L)^{-1/2}(1+\mathcal L)^{1/2}(1+t+t\mathcal{L})^{-1}|(1+t\mathcal{L})^{-1}\right\|_{p\to p}\frac{\,dt}{\sqrt t}\\
&&\le \int_0^1 C\sqrt t\left\|(t+t\mathcal{L})^{1/2}(1+t+t\mathcal{L})^{-1}(1+t\mathcal{L})^{-1}\right\|_{p\to p}\frac{\,dt}{\sqrt t}\\
&&\le \int_0^1 C\left\|(1+t+t\mathcal{L})^{-1/2}(1+t\mathcal{L})^{-1}\right\|_{p\to p}\,dt\\
&&\le C,
\end{eqnarray*}
and
\begin{eqnarray*}
\int_1^\infty\|\nabla (1+t+t\mathcal{L})^{-1}\|_{p\to p}\frac{\,dt}{\sqrt t} &&\le \int_1^\infty\|\nabla(1+\mathcal L)^{-1/2}(1+\mathcal L)^{1/2} (1+t+t\mathcal{L})^{-1}\|_{p\to p}\frac{\,dt}{\sqrt t}\\
&&\le \int_1^\infty C\|(t+t\mathcal{L})^{1/2} (1+t+t\mathcal{L})^{-1}\|_{p\to p}\frac{\,dt}{ t}\\
&&\le \int_1^\infty C\| (1+t+t\mathcal{L})^{-1/2}\|_{p\to p}\frac{\,dt}{ t}\\
&&\le \int_1^\infty C\frac{\,dt}{ t\sqrt{(1+t)}}\le C.
\end{eqnarray*}
The above two estimates imply that
$$\left\|\int_0^1|\nabla (1+t\mathcal{L})^{-1}|\frac{\,dt}{\sqrt t}\right\|_{p\to p}\le C+C\|\nabla(1+\mathcal L)^{-1/2}\|_{p\to p}\le C.$$
\end{proof}

\begin{lem}\label{lemma-s2}
Let $(X,\mu,\E)$ be a Dirichlet metric measure space.
Suppose that for some $p\in (2,\infty)$, $|\nabla (1+\mathcal{L})^{-1/2}|$ is bounded on $L^p(X,\mu)$.
Assume for some $\nu\in [0,1/2)$, it holds for any $t>1$ that
\begin{eqnarray}\label{a3}
\|\nabla (1+t\mathcal{L})^{-1}\|_{p\to p}\le Ct^{-\nu}.
\end{eqnarray}
Then for any $t>1$ it holds
\begin{eqnarray*}
\|\nabla (1+t\mathcal{L})^{-1/2}\|_{p\to p}\le  Ct^{-\nu}.
\end{eqnarray*}
\end{lem}
\begin{proof}
The boundedness of $|\nabla (1+\mathcal{L})^{-1/2}|$ on $L^p(X,\mu)$ implies that for any $r>0$
\begin{equation}\label{a4}
\||\nabla e^{-r\mathcal{L}}|\|_{p\to p}\le \|\nabla (1+\mathcal{L})^{-1/2}\|_{p\to p}\|(r+r\mathcal{L})^{1/2}e^{-r(1+\mathcal{L})}|\|_{p\to p} \frac{e^{r}}{\sqrt r}\le Cr^{-1/2}e^{r}.
\end{equation}

Write
\begin{eqnarray*}
|\nabla(t+t\mathcal{L})^{-1/2}-\nabla(1+t\mathcal{L})^{-1/2}|
&&\le \frac{1}{\pi}\int_0^\infty [e^{-s}-e^{-st}] |\nabla e^{-st\mathcal{L}}|\frac{\,ds}{\sqrt s}\\
&&=\int_0^{1/t}\cdots+\int_{1/t}^\infty\cdots=:I+II.
\end{eqnarray*}
For the term $I$, by using \eqref{a4} and the fact $|e^{-st}-e^{-s}|\le Cst$ for any $t>1$ and $s\le 1/t$, one has
\begin{eqnarray*}
\|I\|_{p\to p}&&\le C\int_0^{1/t} Cst\|\nabla e^{-st\mathcal{L}}\|_{p\to p}\frac{\,ds}{\sqrt s}\le \int_0^{1/t} \frac{Cst e^{st}}{\sqrt{st}}\frac{\,ds}{\sqrt s}\le \frac{C}{\sqrt t}.
\end{eqnarray*}
For the term $II$, one has via \eqref{a3} that
\begin{eqnarray*}
\|II\|_{p\to p}&&\le \int_{1/t}^\infty (e^{-s}-e^{-st})\|\nabla e^{-st\mathcal{L}}\|_{p\to p}\frac{\,ds}{\sqrt s}\\
&&\le \int_{1/t}^\infty e^{-s}\|\nabla (1+st\mathcal{L})^{-1}\|_{p\to p} \|(1+st\mathcal{L})e^{-st\mathcal{L}}\|_{p\to p}\frac{\,ds}{\sqrt s}\\
&&\le \int_{1/t}^\infty e^{-s}(st)^{-\nu}\frac{\,ds}{\sqrt s}\le Ct^{-\nu},
\end{eqnarray*}
where $\nu\in [0,1/2)$. The estimates of $I$ and $II$ imply that
\begin{eqnarray*}
\|\nabla (1+t\mathcal{L})^{-1/2}\|_{p\to p}&&\le \|\nabla (t+t\mathcal{L})^{-1/2}\|_{p\to p}+\|\nabla [(t+t\mathcal{L})^{-1/2}-(1+t\mathcal{L})^{-1/2}]\|_{p\to p}\\
&&\le Ct^{-1/2}+Ct^{-\nu}\le Ct^{-\nu}.
\end{eqnarray*}
\end{proof}

\subsection*{Acknowledgments}
\addcontentsline{toc}{section}{Acknowledgments} \hskip\parindent
R. Jiang was partially supported by NNSF of China (11671039), F.H. Lin
was in part supported by the National Science Foundation Grant DMS-1501000.

\noindent Renjin Jiang \\
\noindent  Center for Applied Mathematics, Tianjin University, Tianjin 300072, China\\
\noindent {rejiang@tju.edu.cn}

\

\noindent Fanghua Lin\\
\noindent Courant Institute of Mathematical Sciences, 251 Mercer Street, New York, NY 10012, USA \\
\noindent linf@cims.nyu.edu

\end{document}